\documentclass[10pt]{article}
\usepackage{amsmath,amsthm,amssymb}

\newtheorem{theorem}{Theorem}[section]

\newtheorem{lemma}[theorem]{Lemma}
\newtheorem{corollary}[theorem]{Corollary}
\newtheorem{remark}[theorem]{Remark}

\newcommand{\refpart}[1]{{\it (#1)}}  

\newcommand{\hpgo}[2]{{}_{#1}\mbox{\rm F}_{\!#2}}
\newcommand{\hpg}[5]{{}_{#1}\mbox{\rm F}_{\!#2}\! \left(\left.{#3 \atop #4}\right|\, #5 \right) }

\renewcommand\ge\geqslant
\renewcommand\geq\geqslant
\renewcommand\le\leqslant
\renewcommand\leq\leqslant

\newcommand{\CC}{\mathbb{C}}

\newcommand{\PP}{\mathbb{P}}

\newcommand{\RR}{\mathbb{R}}
\newcommand{\ZZ}{\mathbb{Z}}
\newcommand{\bta}{S }

\title{Counting derangements and Nash equilibria}

\author{Raimundas Vidunas\footnote{Lab of Geometric \& Algebraic Algorithms,
        Department of Informatics \& Telecommunications,
        National Kapodistrian University of Athens,
        Panepistimiopolis 15784, Greece. E-mail: {\sf rvidunas@di.uoa.gr}.
        Supported by 
        the EU (European Social Fund) and Greek National Fund through 
        the operational program ``Education and Lifelong Learning" of the
        National Strategic Reference Framework, research funding program
        ``ARISTEIA'', project ``ESPRESSO: Exploiting Structure in 
        Polynomial Equation and System Solving in Physical Modeling".}}

\date{}
        
\begin{document}

\maketitle

\begin{abstract}
The maximal number of totally mixed Nash equilibria 
in games of several players equals the number of block derangements,
as proved by McKelvey and McLennan.
On the other hand, counting the derangements is a well studied problem.
The numbers are identified as linearization coefficients for Laguerre polynomials. 
MacMahon derived a generating function for them as an application
of his master theorem. This article relates the algebraic, combinatorial
and game-theoretic problems that were not connected before.
New recurrence relations, hypergeometric formulas and asymptotics
for the derangement counts are derived. An upper bound for the total number
of all Nash equilibria is given.
\end{abstract}

\section{Game-theoretic introduction}

Game theory offers mathematical modeling of strategic decision making.
A central concept is that of {\em Nash equilibrium}: 
it is a combination of strategies of participating players such that no player 
can improve his payoff by unilaterally changing his strategy. The strategies 
can be pure (when a player chooses a single available option) or mixed (when a player 
makes a choice randomly, by assigning probabilities to his options). 
Totally mixed strategies play any available strategy with a non-zero probability. 
A {\em totally mixed Nash equilibrium} (or TMNE for shorthand) 
is a Nash equilibrium where everyone plays a totally mixed strategy.

As we recall in \S \ref{sec:algeo}, finding TMNE leads to a system 
of polynomial equations for the probabilities. 
The polynomial system has finitely many solutions (in $\CC)$ generalically, 
hence games with generic payoffs have finitely many TMNE. 
An interesting question is: 
\begin{quote}
Consider a game of $S$ players, each with $m_1,m_2,\ldots,m_S$ options respectively.
Assuming generic payoffs, what is the maximal possible number of TMNE 
for fixed $S$ and $m_1,m_2,\ldots,m_S$?
\end{quote}
McKelvey and McLennan \cite{McKL97} answered this question 
in terms of a combinatorial count of certain partitions of 
\begin{equation} \label{eq:nn}
N:=m_1+m_2+\ldots+m_S-S
\end{equation}
elements into $S$ sets with $m_j-1$ elements each (\mbox{$j\in\{1,2,\ldots,S\}$}).
The count can be eloquently formulated as follows:
\begin{quote}
Consider a card recreation 
of $S$ players, each with $n_1,n_2,\ldots,n_S$ cards originally.
All cards are shuffled together, and then each player $j$ receives the same number $n_j$ 
of cards as originally. Let $E(n_1,n_2,\ldots,n_S)$ denote the number of ways to deal the cards in 
such a way that no player receives a card that he held originally.

The maximal number of TMNE 
in a generic 
(i.e., regular \cite{McKL97}) game with $m_1,m_2,\ldots,m_S$ pure options equals 
\begin{equation} \label{eq:es}
E(m_1-1,m_2-1,\ldots,m_S-1).
\end{equation}
\end{quote}
If $n_1=\ldots=n_S=1$, these partitions (or permutations) are known as {\em derangements}.
The permutations of $n_1+\ldots+n_S$ elements 
that give their partitions  into $S$ subsets as restricted in the card recreation
are well-studied \cite{EG76}, \cite{AIK78}, \cite{FZ88}
as {\em generalized derangements} (or by a similar term). 
The count of the permutations or partitions differs by the factor $n_1!\cdots n_S!$.
Let us refer to the restricted partitions as {\em block derangements}. 

MacMahon \cite[\S III, Ch. III]{McMh16} gave a generating function for the number 
of block derangements, as one of the first applications of his master theorem. 
Let $\sigma_j$ denote the $j$-th elementary symmetric polynomial in the variables $x_1,x_2,\ldots,x_S$:
\begin{equation}  \label{eq:sigm}
\sigma_1=\sum_{i=1}^S x_i, \qquad \sigma_2=\sum_{i=2}^{S}\sum_{j=1}^{i-1} x_ix_j, 
\qquad  \ldots, \qquad \sigma_S=\prod_{i=1}^S x_i.
\end{equation}
MacMahon proved that $E(n_1,n_2,\ldots,n_S)$ equals the coefficient to 
$x_1^{n_1}x_2^{n_2}\cdots x_S^{n_S}$ in the multivariate Taylor expansion 
of the rational function
\begin{equation} \label{eq:fcard}
F_{\rm M}(x_1,x_2,\ldots,x_S)=
\frac{1}{1-\sigma_2-2\sigma_3-\ldots-(S-1)\sigma_S}
\end{equation}
at $(x_1,x_2,\ldots,x_S)=(0,0,\ldots,0)$. 
The generating function for the maximal number of TMNE is adjusted 
by the argument shift in (\ref{eq:es}). The adjustment is by the factor $\sigma_S$,
thus the maximal number of TMNE
(in a generic game of $S$ players with respectively $m_1,m_2,\ldots,m_S$ pure options) 
equals the coefficient to $x_1^{m_1}x_2^{m_2}\cdots x_S^{m_S}$ 
in the multivariate Taylor expansion of
\begin{equation} \label{eq:fnash}
F_{\rm N}(x_1,x_2,\ldots,x_S)=
\frac{\sigma_S}{1-\sigma_2-2\sigma_3-\ldots-(S-1)\sigma_S}
\end{equation}
at $(x_1,x_2,\ldots,x_S)=(0,0,\ldots,0)$. 

Moreover, the numbers $E(n_1,n_2,\ldots,n_S)$ are identified \cite{EG76}
as {\em linearization coefficients} for the Laguerre polynomials \cite{Wiki}:
\begin{equation*}
L_n(z)=\sum_{k=0}^n {n\choose k} \frac{(-1)^k}{k!} z^k.
\end{equation*}
The Laguerre polynomials form an orthonormal basis in the Hilbert space $L^2([0,\infty);w(z))$,
with respect to the weight function $w(z)=\exp(-z)$. Therefore
\begin{equation*}
\int_0^{\infty} L_n(z)L_m(z)\exp(-z)dz=\delta_{n,m}.
\end{equation*}
The linearization problem  \cite{AIK78}, \cite{FZ88} 
is interested in expressing products of Laguerre polynomials
in the basis of Laguerre polynomials: 
\begin{align*}
L_{n_1}(z)L_{n_2}(z)\cdots L_{n_T}(z)=
\sum_{k=0}^{N} C^{(k)}_{n_1,n_2,\ldots,n_T}\,L_{k}(z).
\end{align*} 
The linearization coefficients $C^{(k)}_{n_1,n_2,\ldots,n_T}$ are naturally computed as 
\begin{align} \label{eq:linc1}
C^{(k)}_{n_1,n_2,\ldots,n_T}
 = \int_0^\infty L_{k}(z)L_{n_1}(z)L_{n_2}(z)\cdots L_{n_T}(z)\exp(-z)dz.
\end{align} 
Remarkably \cite{EG76}, 
\begin{align} \label{eq:linc2}
E(k,n_1,n_2,\ldots,n_T)=\varepsilon C^{(k)}_{n_1,n_2,\ldots,n_T},
\end{align} 
where $\varepsilon=(-1)^{k+n_1+n_2+\ldots+n_T}$.
Zeilberger \cite{Zeilb} tells the interesting story of Gillis investigating the derangement count
in 1928, deriving the same recurrences for the linearization coefficients in 1960,
and noticing the coincidence in 1976. Quite similarly, this article relates the results
in \cite{McKL97} to MacMahon's master theorem and known results on block derangements.

The Nash equilibria, the multihomogeneous B\'ezout bound for their algebraic system
and MacMahon's master theorem are related in \S \ref{sec:algeo}. The supplementing article
\cite{EV14} generalizes the application of MacMahon's master theorem 
to more general algebraic systems, and discusses computational complexity.
Section \ref{sec:special} here summarizes easy special cases of counting block derangements.
Truly new results are presented in \S \ref{sec:recr}--\ref{sec:asympt}:
compact recurrence relations and asymptotics for $E(n_1,\ldots,n_S)$ and an upper bound 
for the total number of all Nash equilibria. 

\section{Applying MacMahon's master theorem}
\label{sec:algeo}

Let $\Omega=\{1,2,\ldots,S\}$ denote the set of players, and let 
$\Theta_j=\{1,2,\ldots,m_j\}$ denote the set of options of the player $j\in\Omega$.  
For any combination $(k_1,k_2,\ldots,k_S)$ of available options  with 
$k_i\in\Theta_i$ each player $j\in\Omega$ 
receives a defined payoff $a^{(j)}_{k_1,k_2,\ldots,k_S}$. 
The totally mixed strategies in a TMNE are then such that each player $j$ receives 
the same payoff $P_j$ with any strategy (as long as others' strategies do not change). 
For $j\in\Omega$, $k\in\Theta_j$, let $p^{(j)}_k$ denote 
the probability that the $j$th player chooses the $k$th option in his strategy. 
The non-zero probabilities must satisfy the algebraic equations:
\begin{align} \label{eq:tmne1}
P_1=& \sum_{k_2,k_3,\ldots,k_S} a^{(1)}_{i,k_2,k_3,\ldots,k_S}\,
p^{(2)}_{k_2}p^{(3)}_{k_3}\cdots p^{(S)}_{k_S},  
\qquad \mbox{for } i\in\Theta_1, \\ 
P_2=& \sum_{k_1,k_3,\ldots,k_S} a^{(2)}_{k_1,i,k_3,\ldots,k_S}\,
p^{(1)}_{k_1}p^{(3)}_{k_3}\cdots p^{(S)}_{k_S},
 \qquad \mbox{for } i\in\Theta_2, \\ 
 & \cdots \nonumber \\  \label{eq:tmnes}
P_S=& \sum_{k_1,\ldots,k_{S-1}} a^{(S)}_{k_1,\ldots,k_{S-1},i}\;
p^{(1)}_{k_1}p^{(2)}_{k_2}\cdots p^{(S-1)}_{k_{S-1}}, 
\quad\, \mbox{for } i\in\Theta_S, \\ 
\label{eq:tmnep}
1=& \; p^{(j)}_1+p^{(j)}_2+\ldots+p^{(j)}_{m_j}, 
\qquad\mbox{for } j\in\Omega. 
\end{align}
This is a system of $m_1+\ldots+m_S+S$ equations in exactly so many variables
(counting $P_1,P_2,\ldots,P_S$ as well). 
If the payoffs $a^{(j)}_{k_1,k_2,\ldots,k_S}$ are generic, the number of solutions is finite.
When the strategies are not totally mixed, there are inequality conditions as we remind 
in \S \ref{sec:allnash}.

We keep the notation $n_j=m_j-1$ for $j\in\Omega$.
As in \cite{McKL97}, we simplify the algebraic system (\ref{eq:tmne1})--(\ref{eq:tmnep}) 
by eliminating $P_1,\ldots,P_S$ in each block (\ref{eq:tmne1})--(\ref{eq:tmnes}),
and ignoring the equations (\ref{eq:tmnep}). This gives us a system of $N=m_1+\ldots+m_S-S$
equations in $m_1+\ldots+m_S$ variables, but the equations are multihomogeneous and
multilinear in the $S$ blocks $p^{(j)}_1,\ldots,p^{(j)}_{m_j}$ \mbox{($j\in\Omega$)} 
of variables. The multihomogeneous system 
on  $\PP^{n_1}\times\cdots\times\PP^{n_S}$ has finitely many complex solutions generically,
since the product of projective spaces has the total dimension $N=n_1+\ldots+n_S$,
Each TMNE 
gives a solution of the multihomogeneous system.
Conversely: a multihomogeneous solution is normalized to a TMNE 
by the equations (\ref{eq:tmnep}) if it is defined over $\RR$ and representative values
$p^{(j)}_1,\ldots,p^{(j)}_{m_j}$ in each block $j\in\Omega$ are either all positive or all negative.

The number of TMNE in generic games 
is bounded by the BKK (Bernstein-Khovansky-Koushnirenko \cite{Bern75}) 
bound of the algebraic system.
The TMNE system (\ref{eq:tmne1})--(\ref{eq:tmnes}) is multihomogenous,
hence the more specific {\em multihomogenous B\'ezout bound} applies.
From a geometric perspective, we have a full intersection on 
$\PP^{n_1}\times\cdots\times\PP^{n_S}$ generically.
The theorem below is formulated for dehomogenized variables.
\begin{theorem}[Multihomogeneous B\'ezout bound]\label{thm:mBez}
Consider a system of $N$ polynomial equations in $N$ affine variables,
partitioned into $S$ subsets so that the $j$-th subset includes
$n_j$ affine variables, and $N=n_1+\cdots +n_S$.
Let $d_{ij}$ be the degree of the $i$-th equation in
the $j$-th variable subset, for $i=1,\dots,N$ and $j=1,\dots,S$.
If the number of complex roots of the polynomial system is finite,
the coefficient of $x_1^{n_1}\cdots x_S^{n_S}$ in
\begin{equation} \label{eq:mhbprod}
\prod_{i=1}^N (d_{i1}x_1 + \cdots + d_{iS}x_S ) 
\end{equation}
is an upper bound for the number of complex roots.
For generic coefficients this bound is tight.
\end{theorem}
\begin{proof}
See \cite{MSW95}, for example.
\end{proof}

\begin{corollary} \label{eq:tmneb}
Let $X=x_1+\ldots+x_S$. The B\'ezout multihomogenous bound 
for the system $(\ref{eq:tmne1})$--$(\ref{eq:tmnes})$ equals the coefficient to
$x_1^{n_1}\cdots x_S^{n_S}$ in the product
\begin{equation} \label{eq:mbx}
 \prod_{j=1}^S (X-x_j)^{n_j}. 
\end{equation}
\end{corollary}
\begin{proof}
Let $P_j(i)$ denote the expression of $P_j$ in (\ref{eq:tmne1})--(\ref{eq:tmnes})
with the specified $i$. The equations 
\begin{align*}
P_1(1)=P_1(2), \ P_1(1)=P_1(3), \ \ldots, \ P_1(1)=P_1(m_1)
\end{align*}
contribute the factor $(X-x_j)^{n_j}$ in (\ref{eq:mhbprod}),
as they do not contain variables of the $j$th block, and are linear
in the variables of each other block.
\end{proof}
The corollary 
begs application of MacMahon's master theorem \cite{McMh16}.
This theorem has most powerful applications to counting restricted 
partitions and proving binomial identities. 
\begin{theorem}[MacMahon, 1916]
Consider  a complex matrix 
\begin{equation} \label{eq:matrix}
A=\left( \begin{array}{cccc}
a_{11} & a_{12} & \cdots & a_{1\bta} \\
a_{21} & a_{22} & \cdots & a_{2\bta} \\
\vdots & \vdots & \ddots & \vdots \\
a_{\bta1} & a_{\bta2} & \cdots & a_{\bta\bta}
\end{array} \right).
\end{equation}
Let $x_1,x_2,\ldots,x_\bta$ be formal variables, 
and let $V$ denote the diagonal matrix with the non-zero entries $x_1,x_2,\ldots,x_\bta$.
The coefficient to $x_1^{n_1}x_2^{n_2}\ldots x_\bta^{n_\bta}$ in
\begin{equation} \label{eq:mhms0}
\prod_{j=1}^{\bta} (a^{}_{j1}x^{}_1+a^{}_{j2}x^{}_2+\ldots+a^{}_{j\bta}x^{}_\bta)^{n_j}
\end{equation}
equals the coefficient to $x_1^{n_1}x_2^{n_2}\ldots x_\bta^{n_\bta}$ in
the multivariate Taylor expansion of
\begin{equation} \label{eq:mhms}
f(x_1,x_2,\ldots,x_\bta)=\frac{1}{\det (\mbox{\rm Id}-VA)}
\end{equation}
around $(x_1,x_2,\ldots,x_\bta)=(0,0,\ldots,0)$.
\end{theorem}

For the context of the TMNE system (\ref{eq:tmne1})--(\ref{eq:tmnes}), 
the matrix $A$ is determined by Corollary \ref{eq:tmneb}:
\begin{equation}
A=\left( \begin{array}{ccccc}
0 & 1 & 1 & \cdots & 1 \\
1 & 0 & 1 & \cdots & 1 \\
1 & 1 & 0 & \cdots & 1 \\
\vdots & \vdots & \ddots & \ddots & \vdots \\
1 & 1 & \cdots & 1 & 0
\end{array} \right).
\end{equation}
Let $M$ denote the matrix $(\mbox{\rm Id}-VA)$. We have
\[
M=\left( \begin{array}{ccccc}
1 & -x_1 & -x_1 & \cdots & -x_1 \\
-x_2 & 1 & -x_2 & \cdots & -x_2 \\
-x_3 & -x_3 & 1 & \cdots & -x_3 \\
\vdots & \vdots & \ddots & \ddots & \vdots \\
-x_S  & -x_S & \cdots & -x_S & 1
\end{array} \right).
\]
The function in (\ref{eq:mhms}) is then the generating function for the numbers $E(n_1,\ldots,n_S)$. 
To prove the generating function (\ref{eq:fcard}), we just have to compute the determinant.
\begin{lemma} 
Let $\sigma_1,\sigma_2,\ldots,\sigma_S$ be the elementary 
symmetric polynomials, as in $(\ref{eq:sigm})$. Then
\[
\det M =1-\sigma_2-2\sigma_3-\ldots-(S-1)\sigma_S.
\]
\end{lemma}
\begin{proof}
The determinant is a symmetric function of $x_1,x_2,\ldots,x_S$, 
at most linear in each variable. Hence it is a linear combination of 
$\sigma_0=1$ and $\sigma_1,\sigma_2,\ldots,\sigma_S$.
The linear combination can be recovered from the diagonal specialization
$x_1=x_2=\ldots=x_S$.
If we set all variables equal to $1/\lambda$, 
\[
\det M=\frac1{\lambda^S} \,\det(\lambda\;\mbox{\rm Id}-A).
\]
Here $\det(\lambda\;\mbox{\rm Id}-A)$ is the characteristic polynomial of $A$.
The rank of $(\mbox{\rm Id}+A)$ equals 1, hence $\lambda=-1$ is an eigenvalue of $A$ 
with the multiplicity $S-1$. Other eigenvalue is $\lambda=S-1$, with an eigenvector consisting of all 1's.
Hence $\det(\lambda\;\mbox{\rm Id}-A)=(\lambda+1)^{S-1}\,(\lambda-S+1)$ and
\begin{equation} \label{eq:mev}
\det M=(1+x_1)^{S-1}\,\left(1-(S-1)x_1\right)=\sum_{j=0}^S(1-j){S\choose j}x_1^j
\end{equation}
when $x_1=x_2=\ldots=x_S$. 
For each $j\in\Omega$, 
the term with $x_1^j$ represents ${S\choose j}$ summands of $\sigma_j$. 
Without the diagonal specialization, {$\det M=\sum_{j}(1-j)\sigma_j$} as claimed.
\end{proof}

The following variation will be used in \S \ref{sec:asympt}.
\begin{lemma} \label{th:edet}
$ \displaystyle
\det \left( \begin{array}{cccc}
1+x_1 & 1 & \cdots & 1 \\
1 & 1+x_2 & \cdots & 1 \\
\vdots & \vdots &  \ddots & \vdots \\
1  &1 & \cdots & 1+x_T
\end{array} \right) =\sigma_T+\sigma_{T-1}.
$
\end{lemma}
\begin{proof}
After the specialization $x_1=x_2=\ldots=x_T=y$, 
the determinant equals $y^{T-1}(y+T)$.
\end{proof}

To relate the product in (\ref{eq:mbx}) to the card recreation, 
we write it down as a product of $N$ terms
\[
(x_2+x_3+\ldots+x_{S}) \,\cdots\, (x_1+x_2+\ldots+x_{S-1}).
\]
Let the first term represent the players to which the first card of the first player could be dealt, and so on.
Expansion of the product gives the generating function for the number of ways the cards could be dealt
(in variable quantities to the players) so that no player receives a card of his own. 
The coefficient to $x_1^{n_1}\cdots x_S^{n_S}$ gives the McKelvey-McLennan count.

In \cite{McKL97}, the combinatorial count was derived using the more general BKK bound.
As detailed in \cite{EV14}, the BKK bound is defined in terms of {\em mixed volumes},
and is related to the multihomogeneous B\'ezout bound via matrix {\em permanents} \cite{Wiki}.
The supplementing article \cite{EV14} explores application of MacMahon's master theorem
to BKK bounds of more general algebraic systems.

Generally, not all complex solutions of  (\ref{eq:tmne1})--(\ref{eq:tmnes})
give proper real solutions defining TMNE. But \cite[\S 4]{McKL97} gives families of games 
where the number of TMNE achieves the BKK 
bound. In \cite{JPS09}, the systems 
whose all roots represent TMNE are parametrized. 

\section{Special cases}
\label{sec:special}

Here we summarize known explicit results about the numbers $E(n_1,\ldots,n_S)$ counting
block derangements, and give a complete set of hypergeometric formulas for the case $S=3$. 

For $S=2$ players, $F_{\rm M}(x_1,x_2)=1/(1-x_1x_2)$. This is consistent with 
the orthogonality of Laguerre polynomials, and with the fact 
that there is at most one TMNE 
in generic 
games of two players each with $m_1=m_2$ choices, 
while there are generically no TMNE 
if $m_1\neq m_2$. 
The system (\ref{eq:tmne1})--(\ref{eq:tmnep}) is linear when $S=2$,
and has an over-determined subsystem if additionally $m_1\neq m_2$.

Similarly, $E(n_1,n_2,\ldots,n_S)=0$ for $n_1>n_2+\ldots+n_S$
both by the McKelvey-McLennan count and 
in the generating function (as the denominator has no linear terms).
In the context of Laguerre polynomials, this reflects orthogonality of $L_{n_1}(z)$
to the product $L_{n_2}(z)\cdots L_{n_S}(z)$ of lower degree.
For \mbox{$n_1=n_2+\ldots+n_S$}, we count dealings of $n_1$ cards of the first player to the others: 
\begin{equation}
E(n_1,n_2,\ldots,n_S)=\frac{n_1!}{n_2!n_3!\cdots n_S!}.
\end{equation}

Table 1.1 in \cite{McKL97} gives some numbers $E(n,n,\ldots,n)$.
A few columns and rows of this table appear in Sloan's encyclopedia of integer sequences
\cite{Sloan}. In particular,
\begin{itemize}
\item $E(1,1,\ldots,1)$ equals the number of {\em derangements} of $S$ elements 
(OEIS sequence A000166),  i.e.,  the number of permutations without fixed elements. 
We have
\begin{equation}
E(1,1,\ldots,1)=S! \, \sum_{j=0}^S \frac{(-1)^j}{j!}.
\end{equation}
This is the maximal number of TMNE 
for $S$ players when each has 2 options.
\item $E(2,2,\ldots,2)$, $E(3,3,\ldots,3)$, $E(4,4,\ldots,4)$, $E(5,5,\ldots,5)$ 
are the {\em card-matching numbers}  (or {\em dinner-dinner matching numbers}) 
defined by following the McKelvey-McLennan description.
These are the OEIS sequences A000459 (alias A059072), A059073, A059074 and A123297.
\item $E(n,n,n)$ are the {\em Franel numbers} (OEIS sequence A000172):
\begin{equation} \label{eq:franel}
E(n,n,n)=\sum_{i=0}^{n} {n \choose i}^{\!3}.
\end{equation}
In the context of Laguerre polynomials, this was noticed by Askey; 
see end comment \refpart{b} in \cite{EG76}. 
\end{itemize}

The linearization problem for orthogonal polynomials is substantially  solved 
when  products of two polynomials are linearized. 
Hence the numbers $E(n_1,n_2,n_3)$ are of foremost interest. 
For $S=4$, a linearization reduction leads to
\begin{equation} \label{eq:opsum}
E(n_1,n_2,n_3,n_4)=\sum_{k} E(k,n_1,n_2)\,E(k,n_3,n_4),
\end{equation}
as asserted in \cite[(28)]{EG76}.
More generally,
\begin{equation}
E(n_1,\ldots,n_S,m_1,\ldots,m_T)=\sum_{k} E(k,n_1,\ldots,n_S)E(k,m_1,\ldots,m_T).
\end{equation}
For computational purposes, the integral representation (\ref{eq:linc1})--(\ref{eq:linc2})
gives an effective, polynomial time algorithm to get the $E$-numbers.

Hypergeometric $\hpgo32(\pm1)$ expressions for $E(n_1,n_2,n_3)$ are known \cite{EG76}.
Recall the definition of hypergeometric series:
\begin{equation}
\hpg32{\alpha,\beta,\gamma}{\zeta,\,\eta}{z} = \sum_{n=0}^\infty
\frac{(\alpha)_n(\beta)_n(\gamma)_n}{(\zeta)_n(\eta)_n\;n!}z^n,
\end{equation}
where $(\alpha)_n=\alpha(\alpha+1)(\alpha+2)\cdots(\alpha+n-1)$ is the rising factorial. 
The hypergeometric series terminates if (at least) one of the parameters $\alpha,\beta,\gamma$
is zero or a negative integer. It is not well defined if $\zeta$ or $\eta$ is zero or a negative integer,
unless some of $\alpha,\beta,\gamma$ are such integers closer to $0$.
For cleaner formulas, we change the notation from $(n_1,n_2,n_3)$ to $(a,b,c)$, 
and set
\begin{equation}  \label{eq:pqrdef}
p=\frac{a+b+c}2, \qquad q=\frac{a+b+c-1}2, \qquad r=\left\lfloor \frac{a+b+c}2 \right\rfloor.
\end{equation}
For nonzero $E(a,b,c)$, the triangle inequalities $a\le b+c$, $b\le a+c$, $c\le a+b$ must be satisfied.
This is equivalent to $r\ge\max(a,b,c)$. 
\begin{theorem}
If $c=\max(a,b,c)$, then
\begin{align} \label{eq:bf0} 
E(a,b,c) &= \sum_{k=0}^{a+b-c} 
{a\choose k}{b\choose c-a+k}{c\choose b-k} \\ \label{eq:f0} 
&= \frac{c!}{(a+b-c)!(c-a)!(c-b)!}\,\hpg32{c-a-b,-a,-b}{c-a+1,c-b+1}{\!-1} \\  \label{eq:f1} 
&= \frac{2^{a+b-c}\,c!}{(a+b-c)!(c-a)!(c-b)!}\,\hpg32{c-p,\,c-q,\,c+1}{c-a+1,c-b+1}{1}. 
\end{align}
\end{theorem}
\begin{proof} These are the formulas (35)--(37) in \cite{EG76}. 
Note that the first formula directly specializes to the Franel numbers (\ref{eq:franel}).
The second formula is the same sum in a hypergeometric form.
Formulas (\ref{eq:f0}) and (\ref{eq:f1}) are related by
Whipple's quadratic transformation \cite[\S 7]{Whp27}:
\begin{equation*}
\hpg32{\alpha,\,\beta,\,\gamma}{\!1\!+\!\alpha\!-\!\beta,1\!+\!\alpha\!-\!\gamma}{z} = (1-z)^{-\alpha} 
\hpg32{\!\frac{\alpha}2,\frac{\alpha+1}2,1\!+\!\alpha\!-\!\beta\!-\!\gamma}
{1\!+\!\alpha\!-\!\beta,1\!+\!\alpha\!-\!\gamma}{\!\frac{-4z}{(1-z)^2}}\!.
\end{equation*}
To prove the theorem independently from \cite{EG76}, it is thus enough to show (\ref{eq:f1}).
We use the generating function:
\begin{align*}
F_{\rm M}(x,y,z) & =\frac{1}{1-(xy+xz+yz+2xyz)}
=\sum_{n=0}^{\infty} x^ny^nz^n\left(2+\frac1x+\frac1y+\frac1z\,\right)^{\!n} \\
& = \sum_{n=0}^{\infty} \, \sum_{a,b,c\ge0}^{a+b+c\le n} 
\frac{n!\;2^{n-a-b-c}}{a!\,b!\,c!\,(n\!-\!a\!-\!b\!-\!c)!}\,x^{n-a}y^{n-b}z^{n-c} \\
& = \sum_{n=0}^{\infty} \; \sum_{a+b+c\ge 2n} ^{a,b,c\le n}
\frac{n!\;2^{a+b+c-2n}\;x^{a}\,y^{b}\,z^{c}}{(n-a)! (n-b)! (n-c)! (a+b+c-2n)!} \\
& = \sum_{a,b,c=0} ^{\infty} \; \sum_{n=\max(a,b,c)}^{r} 
\frac{n!\;2^{a+b+c-2n}\;x^{a}\,y^{b}\,z^{c}}{(n-a)! (n-b)! (n-c)! (a+b+c-2n)!}.
\end{align*}
The upper summation limit $r$ is defined in (\ref{eq:pqrdef}).
We assumed $c=\max(a,b,c)$.
Shifting the inner summation index $n$ by $c$ gives (\ref{eq:f1}).
\end{proof}

More $\hpgo32(1)$ expressions for $E(a,b,c)$ are obtained by using Whipple's
group of transformations for $\hpgo32(1)$ series \cite{Whp24}. 
Whipple defined an orbit 120 {\em allied} closely related general $\hpgo32(1)$ series,
analogous to the 24 Kummer's $\hpgo21(z)$ functions. 
When terminating \cite{RJRJ97} or ill-defined $\hpgo32(1)$ sums are involved, 
Whipple's relations between the 120 allied functions degenerate. 
Starting from (\ref{eq:f1}), we obtain both ill-defined
and terminating (of various length) $\hpgo32(1)$ sums.

The most interesting new expressions for $E(a,b,c)$ are summarized in the following theorem.
The terminating sums have $a,b,c,r-a,r-b$ or $r-c$ terms.  
In particular, formula (\ref{eq:f6}) has $\min(p-b,p-c)$ terms for even $a+b+c$,
while $a$ terms for odd $a+b+c$. 
The similar formula  (\ref{eq:f5b}) has $p-c$ terms for even $a+b+c$, 
but is undefined for odd $a+b+c$.
Other terminating $\hpgo32(1)$ expressions are obtained by rewriting 
the presented sums in the reverse order. 
For example, formula (\ref{eq:f1}) is rewritten as follows, for even and odd $a+b+c$ respectively:
\begin{align} \label{eq:f1a}
E(a,b,c) &=\frac{p!}{(p-a)!(p-b)!(p-c)!}\,\hpg32{a-p,b-p,c-p}{-p,\;\frac12}{1} \\
\label{eq:f1b} &=\frac{2\cdot q!}{(q-a)!(q-b)!(q-c)!}\,\hpg32{a-q,b-q,c-q}{-q,\;\frac32}{1}. 
\end{align} 
\begin{theorem}
The following terminating $\hpgo32(1)$ expressions hold:
\begin{align} \label{eq:strehl}
E(a,b,c)&=  {c\choose b}{2b\choose a+b-c}\,\hpg32{\!c-p,\,c-q,-b}{c-b+1,\,\frac12-b}{1} 
\qquad \mbox{(if $c\ge b$)} \\ 
 \label{eq:sun} &= \frac{2^{a+b+c}\,(\frac12)_a\,(\frac12)^{}_b\;c!}{(a+b-c)!(a-b+c)!(b-a+c)!}
\,\hpg32{\!c-p,\,c-q,\frac12}{\frac12-a,\,\frac12-b}{1} \\  
\label{eq:f8} &= \frac{(a+b+c)!}{(a+b-c)!(a-b+c)!(b-a+c)!}\,\hpg32{\!-a,-b,-c}{-p,-q}{1} \\  
\label{eq:f6} &= {p\choose a}{2a\choose a+b-c}\,\hpg32{\!-a,\,c-p,\,b-p}{-p,\,\frac12-a}{1} \\ 
\label{eq:f7} &= {q\choose a}{2a\choose a+b-c}\,\hpg32{\!-a,\,c-q,\,b-q}{-q,\,\frac12-a}{1}. 
\end{align}
For even $a+b+c$:
\begin{align}  
\label{eq:f4a} E(a,b,c)
&= {2a \choose a+b-c}\frac{b!\,c!}{a!\,(p-a)!^2}\,\hpg32{c-p,\,b-p,\,\frac12}{p-a+1,\,\frac12-a}{1} \\ 
 \label{eq:f5b} &= \frac{(-1)^{p-c}\,p!}{(p-a)!(p-b)!(p-c)!}\,\hpg32{c-p,-a.-b}{-p,\,c-q+1}{1}. 
\end{align}
For odd $a+b+c$:
\begin{align}
\label{eq:f4b} E(a,b,c)
&= {2a \choose a+b-c}\frac{b!\,c!}{a!\,(q-a)!(q-a+1)!}\,\hpg32{c-q,b-q,\frac12}{\!q-a+1,\frac12-a}{1} \\ 
\label{eq:f5a} &= \frac{(-1)^{q-c}\,q!}{(p-a)(q-a)!(q-b)!(q-c)!}\,\hpg32{c-q,-a.-b}{-q,\,c-p+1}{1}. 
\end{align}
\end{theorem}
\begin{proof}
Whipple's symmetries \cite{Whp24}  of $\hpgo32(1)$ functions are summarized as follows. 
Let $r_0,r_1,\ldots,r_5$ be six complex numbers that sum up to 0. 
Similarly to Whipple's notation $Fp(0;4,5)$, we introduce 
\begin{equation*}
F^{+0}_{45}=
\hpg32{r_1+r_4+r_5+\frac12,r_2+r_4+r_5+\frac12,r_3+r_4+r_5+\frac12}{1-r_0+r_4,\;1-r_0+r_5}{1}.
\end{equation*}
For distinct $i,j,k\in\{0,1,\ldots,5\}$, let $F^{+i}_{jk}$ denote the function obtained 
by a corresponding permutation of the $r_j$'s ($r_0\mapsto r_i$, etc.) 
Let $F^{-i}_{jk}$ be the function obtained by multiplying all $r_j$'s by $-1$. 
In total, we have 120 allied functions $F^{+i}_{jk}$, $F^{-i}_{jk}$.
Any three of them are related by a linear relation. In particular, the function
\begin{equation} \label{eq:wsym}
\frac{F^{+0}_{45}}{\Gamma(1-r_0+r_4)\Gamma(1-r_0+r_5)\Gamma(r_1+r_2+r_3+\frac12)}
\end{equation}
is invariant under the permutations of $r_1,r_2,r_3,r_4,r_5$, generally. 
The same $S_5$-symmetry generally holds for other $F^{\pm i}_{jk}$. 
The six $r_j$'s are identified as
\begin{align}
(r_0,r_4,r_5) &=\left( \frac{a+b+c}{3}+\frac12,-\frac{a+b+c}{6}-\frac12, -\frac{a+b+c}{6} \right), \nonumber\\ 
(r_1,r_2,r_3) &=\left( \frac{b+c-2a}{3},\frac{a+c-2b}{3},\frac{a+b-2c}{3}\right).
\end{align}
Particularly,
\begin{align*} \textstyle
r_1+r_4+r_5+\frac12=-a,\quad r_1+r_2+r_4+\frac12=p-c, \quad 1-r_3+r_2=1+c-b,
\end{align*}
etc. Formulas (\ref{eq:strehl})--(\ref{eq:f5a}) are identified and proved by 
checking that the symmetries (\ref{eq:wsym}) hold (in continuous limit)
between well-defined terminating $F^{\pm i}_{jk}$ 
with any fixed upper parameter $\pm i$, and relating the functions with different $\pm i$
by reversal of terminating hypergeometric sums.

\begin{table}
\begin{tabular}{@{}c@{}cccc@{}} 
\\ \hline
Formula & Whipple & \multicolumn{2}{c}{Reversed sums} & Ill-defined \\
\cline{3-4} & functions & for $p\in\ZZ$ & for $q\in\ZZ$ & ``terminating" \\ \hline 
(\ref{eq:f1}) & $F^{+3}_{12}$ &  $F^{-4}_{05}$ (\ref{eq:f1a}) 
& $F^{-5}_{04}$ (\ref{eq:f1b})  & $F^{+1}_{23}$, $F^{+2}_{12}$ \\
(\ref{eq:strehl}) & $F^{-2}_{03}$,  $F^{-1}_{03}$, $F^{-1}_{02}$ & 
$F^{+5}_{14}$, $F^{+5}_{24}$, $F^{+5}_{34}$ &  $F^{+4}_{15}$, $F^{+4}_{25}$, $F^{+4}_{35}$
&  $F^{-3}_{02}$, $F^{-3}_{01}$, $F^{-2}_{01}$ \\
(\ref{eq:sun}) & $F^{+0}_{12}$,  $F^{+0}_{13}$, $F^{+0}_{23}$ & 
$F^{-4}_{35}$, $F^{-4}_{25}$, $F^{-4}_{15}$ &  $F^{-5}_{34}$, $F^{-5}_{24}$, $F^{-5}_{14}$ & --- \\
(\ref{eq:f8}) & $F^{+0}_{45}$ &  \multicolumn{2}{c}{$F^{-1}_{23}$} & $F^{-2}_{13}$, $F^{-3}_{12}$ \\
(\ref{eq:f6}) & $F^{+0}_{14}$,  $F^{+0}_{24}$, $F^{+0}_{34}$ & 
$F^{-2}_{35}$, $F^{-1}_{35}$, $F^{-1}_{25}$ &  $F^{-5}_{23}$, $F^{-5}_{13}$, $F^{-5}_{12}$ 
& $F^{-3}_{25}$, $F^{-3}_{15}$, $F^{-2}_{15}$ \\
(\ref{eq:f7}) & $F^{+0}_{15}$,  $F^{+0}_{25}$, $F^{+0}_{35}$ & 
$F^{-4}_{23}$, $F^{-4}_{13}$, $F^{-4}_{12}$ &  $F^{-2}_{34}$, $F^{-1}_{34}$, $F^{-1}_{24}$ 
& $F^{-3}_{24}$, $F^{-3}_{14}$, $F^{-2}_{14}$ \\
(\ref{eq:f4a}) & $F^{-1}_{05}$,  $F^{-2}_{05}$, $F^{-3}_{05}$ & 
$F^{+3}_{24}$, $F^{+3}_{14}$, $F^{+2}_{14}$ & --- & $F^{+2}_{34}$, $F^{+1}_{34}$, $F^{+1}_{24}$ \\
(\ref{eq:f4b}) & $F^{-1}_{04}$,  $F^{-2}_{04}$, $F^{-3}_{04}$ & --- & 
$F^{+3}_{25}$, $F^{+3}_{15}$, $F^{+2}_{15}$ & $F^{+2}_{35}$, $F^{+1}_{35}$, $F^{+1}_{25}$ \\
(\ref{eq:f5b}) & $F^{-4}_{03}$,  $F^{-4}_{02}$, $F^{-4}_{01}$ & 
$F^{+5}_{12}$, $F^{+5}_{13}$, $F^{+5}_{23}$ & --- & $F^{+2}_{35}\ldots,F^{+3}_{25}\ldots$ \\
(\ref{eq:f5a}) & $F^{-5}_{03}$,  $F^{-5}_{02}$, $F^{-5}_{01}$ & --- &
$F^{+4}_{12}$, $F^{+4}_{13}$, $F^{+4}_{23}$ & $F^{+2}_{34}\ldots,F^{+3}_{24}\ldots$ \\
\hline\\
\end{tabular}
\caption{The allied $\hpgo32(1)$ terminating sums}
\label{tb:whipple}
\end{table}

The whole picture of the allied $\hpgo32(1)$ sums with a non-negative upper parameter 
is given in Table \ref{tb:whipple}. For a concrete identification, we assume $a\le b\le c$. 
Some of  formulas (\ref{eq:strehl})--(\ref{eq:f5a})
are valid for several orderings of $(a,b,c)$, giving three functions in the second column.
The functions in the last column appear as formal reversal of hypergeometric sums
with several upper parameters equal to a non-negative integer, such as (\ref{eq:f1a})--(\ref{eq:f1b}).
The reversed sums have the opposite sign of $\pm i$, and the sets of three parameters 
are complementary. 
The ill-defined functions related to (\ref{eq:f5b}), (\ref{eq:f5a}) are related to  (\ref{eq:f4a}), (\ref{eq:f4b})
as well, but in the context of different parity of $a+b+c$.  
In total, we have 26 functions in the second column, 29 reversed sums, 
and 19 new ill-defined sums. 
Of the remaining 36 allied functions, proper gamma-multiples 
of $F^{+4}_{05}$ and $F^{+5}_{04}$ converge to $E(a,b,c)$.
All 10 functions $F^{-0}_{jk}$ are well-defined convergent series,
though their value is apparently not related to $E(a,b,c)$. 
The other $\hpgo32(1)$ series are ill-defined or divergent,
including the non-terminating cases of (\ref{eq:f4a})--(\ref{eq:f5a}).
\end{proof}

Formulas (\ref{eq:strehl})--(\ref{eq:f8}) and (\ref{eq:f1}) specialize to
the following expressions for the Franel numbers, respectively:
\begin{align} \label{eq:franel1}
E(n,n,n)=\sum_{k=0}^{n} {n \choose k}^{\!3}
 &=\sum_{k=\lceil n/2\rceil}^n {n \choose k}^{\!2}{2k \choose n} \\ \label{eq:franel2}
 &=\frac{1}{2^n}\sum_{k=\lceil n/2\rceil}^n {2k \choose n}{2k \choose k}{2n-2k \choose n-k}
 \quad \\ \label{eq:franel3}
 &= \sum_{k=0}^{n} {n+2k \choose 3k}  {2k \choose k}  {3k \choose k} (-4)^{n-k} \\ \label{eq:franel4}
 & =\sum_{k=0}^{\lfloor n/2\rfloor} {n+k \choose 3k}  {2k \choose k}  {3k \choose k} \, 2^{n-2k}.
\end{align}
The expression in (\ref{eq:franel1}) was proved by Strehl \cite[(29)]{Str94},
while expressions (\ref{eq:franel2}), (\ref{eq:franel3}) were proved by Sun \cite{Sloan}, \cite{Sun13}. 

The symmetric case with 4 arguments is more complicated.
\begin{theorem} \label{eq:nnnn}
$\displaystyle E(n,n,n,n) = \sum_{i=0}^n \sum_{j=0}^n {2i+2j \choose 2i} {n\choose i}^{\!2} {n\choose j}^{\!2}.$
\end{theorem}
\begin{proof}
Formulas (\ref{eq:opsum}), (\ref{eq:strehl}) give $E(n,n,n,n)=\sum_{k=0}^{2n} E(k,n,n)^2$ and
\[
E(k,n,n)={2n \choose k} \, \hpg32{-\frac{k}2,-\frac{k-1}2,-n}{1,\,\frac12-n}{1}
=\sum_{j=0}^{\lfloor k/2\rfloor} {n\choose j}^{\!2} {2n-2j\choose 2n-k}.
\] 
We reindex $j\mapsto n-j$ and then $k\mapsto 2n-k$ to get
\begin{align*}
E(2n-k,n,n) &= \sum_{j=\lceil k/2\rceil}^n {n\choose j}^{\!2} {2j\choose k}, \\
E(n,n,n,n) & = 
\sum_{i=0}^n \sum_{j=0}^n {n\choose i}^{\!2} {n\choose j}^{\!2} 
\;\, \sum_{k=0}^{\min(2i,2j)} {2i\choose k} {2j\choose k}.
\end{align*}
The inner-most sum equals $2i+2j\choose 2i$ by Chu-Vandermonde formula \cite{Wiki}.
\end{proof}

\section{Recurrence relations}
\label{sec:recr}

In \cite{EG76}, recurrence relations for the numbers $E(n_1,\ldots,n_S)$ are derived
from the three-term recurrence for Laguerre polynomials. 
Formulas (31), (32) in \cite{EG76} tell:
\begin{align}
E(1,a,b,c) = & \, (a+1)E(a+1,b,c)+2aE(a,b,c)  +aE(a-1,b,c), \\
2(b-a)E(a,b,c) = & \, (a+1)E(a+1,b,c)+aE(a-1,b,c) \nonumber \\ \label{eq:gillis5}
& -(b+1)E(a,b+1,c)-bE(a,b-1,c).
\end{align}
The argument $c$ can be replaced by any sequence of arguments. 
Formula (34) in \cite{EG76} is a 4-term linear recurrence for $E(a,b,c)$, 
but the $\hpgo32(1)$ expressions imply that any 3 values of $E(a,b,c)$ 
are related by a linear relation. The next theorem spells out basic
3-term relations for $E(a,b,c)$.

Generally,  compact recurrences can be obtained and proved by considering
partial differential equations for the generating function $F_{\rm M}(x_1,\ldots,x_n)$. 
Let us introduce the differential operators 
\[
D_j=x_j\frac{\partial}{\partial x_j}, \qquad \mbox{for $j\in\Omega$.}
\]
Partial differential equations are represented by differential operators  in the Weyl algebra 
$\CC\langle x_1,\ldots,x_S;D_1,\ldots,D_S\rangle$, 
with the nontrivial commutation relations $D_jx_j=x_jD_j+x_j$.
Recurrences are represented by (negative) shift operators in the algebra 
$\CC\langle n_1,\ldots,n_S;T_1,\ldots,T_S \rangle$, 
with the nontrivial commutation relations $T_jn_j=(n_j-1)T_j$. 
The correspondence between the differential and shift operators is realized by the algebra
isomorphism $x_j\mapsto T_j$, $D_j\mapsto n_j$. To get recurrences with fewer terms, 
we look for differential operators (of any order) with few distinct $\CC[x_1,\ldots,x_S]$
monomials in the coefficients to products of $D_j$'s.

\begin{theorem}  \label{lm:rec}
The following recurrences hold:
\begin{align} \label{eq:rec3a}
& \hspace{-250pt} 2(a-b) E(a,b,c) + (a-b+c+1) E(a+1,b,c) \nonumber \\ 
 +(a-b-c-1)E(a,b+1,c) &=0,\\ \label{eq:rec3b}
& \hspace{-250pt} 2a\,E(a-1,b,c)+(a-b+c) E(a,b,c) \nonumber \\
 +(c-a-b-1)E(a,b+1,c) &=0, \\ \label{eq:rec3c}
 & \hspace{-250pt} (a-b)(a+b-c)E(a,b,c)+a(a-b-c-1)E(a-1,b,c) \nonumber \\
 +b(a-b+c+1)E(a,b-1,c) &=0,\\ \label{eq:rec3d}
& \hspace{-250pt} (a-b+c+1)(a+b-c+1) E(a+1,b,c) \nonumber \\
+(3a^2+a-(2a+1)(b+c)-(b-c)^2) E(a,b,c) \nonumber \\
+2a(a-b-c-1)E(a-1,b,c) &=0.
\end{align}
\end{theorem}
\begin{proof}
The first two recurrences follow from the differential operators
\begin{align*}
(x_2-x_1) D_3+(2x_1x_2+x_1+x_2)(D_1-D_2),\\
(x_2+1)(D_3-D_2)+(2x_1x_2+x_2-1)D_1+2x_1x_2 \;
\end{align*}
annihilating $F_{\rm M}(x_1,x_2,x_3)$. 
The  other two recurrences are linear combinations of (\ref{eq:rec3a})--(\ref{eq:rec3b}).
\end{proof}

For any $S\ge 3$, the following two recurrences follow from linear differential operators
for $F_{\rm M}(x_1,\ldots,x_n)$. 
To present the results compactly, we indicate only the shifted parameters of $E=E(n_1,\ldots,n_S)$. 
The first formula is a generalization of the 5-term relation in (\ref{eq:gillis5}).
\begin{theorem} \label{th:rec5}
The following recurrence relations hold:
\begin{align} \label{eq:rec5}
2(n_2-n_1)\, E=&\,(n_1+1)\,E(n_1+1)+n_1\,E(n_1-1) \nonumber\\
& -(n_2+1)\,E(n_2+1)-n_2\,E(n_2-1),\\
(n_1+1)\,E(n_1+1)=&\, n_2\,E(n_2-1)+\ldots+n_S\,E(n_S-1) \nonumber\\
& +(n_2+\ldots+n_S-n_1)\,E.
\end{align}
\end{theorem}
\begin{proof} Let us denote:
\begin{align*}
H&=(1+x_1)(1+x_2)\cdots(1+x_S), \\
G&=(1-D_1-D_2-\ldots-D_S)H,
\end{align*}
so that $F_{\rm M}=1/G$. We have
\begin{align*}
D_jH &=\frac{x_j}{x_j+1}H, \qquad 
G=\left(1-\frac{x_1}{x_1+1}-\ldots-\frac{x_S}{x_S+1}\right)H, \\
D_jG & =\frac{x_j}{x_j+1}G-\frac{x_j}{(x_j+1)^2}H,
\end{align*}
and subsequently
\begin{align*}
H&=-\frac{(x_j+1)^2}{x_j}D_jG+(x_j+1)G \qquad \mbox{(with } j\in\{1,2\}) \\
&=-(x_1+1)D_1G-\ldots-(x_S+1)D_SG+(x_1+\ldots+x_S+1)G.
\end{align*}
Elimination of $H$ gives the following differential operators that annihilate $G$:
\begin{align*}
L_1=&\, -\frac{(x_1+1)^2}{x_1}\,D_1+\frac{(x_2+1)^2}{x_2}\,D_2+(x_1-x_2),\\
L_2=&\, \frac{x_1+1}{x_1}D_1-(x_2+1)D_2-\ldots-(x_S+1)D_S+(x_2+\ldots+x_S).
\end{align*}
To get differential operators annihilating $F_{\rm M}$, multiply each $D_j$ by $-1$ in $L_1,L_2$.
This leads to the claimed recurrences.
\end{proof}
The last recurrences have 5 and $S+1$ terms, respectively. 
But the order of recurrences appears to grow quadratically with $S$.
For example, for $S=4$ 
the recurrence with shifts in $a$ alone has order 6 (thus 7 terms), 
with the coefficients are of degree 9 in $a,b,c,d$.
The multi-variate Zeilberger summation routine \cite{ZeilbMZ}
applied to Theorem \ref{eq:nnnn} returns %
a recurrence of order 6 for $E(n,n,n,n)$, of degree 28 in $n$ 
(in 30 min.~on Maple 14, on a 2.66 GHz Intel Core 2 Duo Mac).
Here is a relatively compact 6-term relation:
\begin{align*}
\big((a\!-\!b)\big(a^2\!+\!2ab\!-\!b^2\!+\!4a\!+\!2-(c\!-\!d)^2\big)-2(b\!+\!1)^2(c\!+\!d\!+\!2)\big)
E(a\!+\!1, b\!+\!1, c, d) \\
+(a\!+\!1)\big((a\!-\!b)(3a\!+\!5b\!+\!7)-(2a\!+\!2b\!+\!3)(c\!+\!d\!+\!2)-(c\!-\!d)^2\big)
E(a, b\!+\!1, c, d) \\
+2a(a+1)(a-b-c-d-2) E(a\!-\!1, b\!+\!1, c, d) \\
+2(b+1)(a+2)(a-b+c+d+2) E(a\!+\!2, b, c, d) \\
+(b\!+\!1)\big((a\!-\!b)(9a\!-\!b\!+\!11)+(6a\!-\!2b\!+\!7)(c\!+\!d\!+\!2)+(c\!-\!d)^2\big) 
E(a\!+\! 1, b, c, d) \\
+2(a+1)(b+1)(5a-5b+c+d+2)E(a,b,c,d)& =0. 
\end{align*}
A few more 6-term relations are obtained by combination with (\ref{eq:rec5}).

\section{Bounding the number of all Nash equilibria}
\label{sec:allnash}

As in \S \ref{sec:algeo}, consider a generic game of $S$ players with $m_1,m_2,\ldots,m_S$ options.
In a {\em subgame} we allow each player $j\in\Omega$ to choose from a fixed non-empty subset 
of his original pure options $\Theta_j$, and keep the payoffs $a^{(j}_{k_1,k_2,\ldots,k_S}$ 
the same for the still possible combinations of pure options. 
Any Nash equilibrium of the original game can be considered as a TMNE of the subgame
that allows only the options played with a non-zero probability. On the other hand, 
not all TMNE of a subgame would be Nash equilibria for the original game. 
In the notation of the proof of Corollary \ref{eq:tmneb}, we must have the inequalities 
$P_j(i)\le P_j(k)$ for any $j,i,k$ with $p^{(j)}_{i}=0$, $p^{(j)}_{k}>0$.

To bound the number of all Nash equilibria (in a generic game),
we add up the maximal $E$-numbers of TMNE for all of its subgames.
Let us denote this sum by $B(m_1,m_2,\ldots,m_S)$.
In the case of two players, $E(i,k)=\delta_{i,k}$, hence
we are counting then the number of pairs of non-empty subsets 
of $\Theta_1,\Theta_2$ of the same size $k$:
\begin{align} \label{eq:max2}
B(m_1,m_2)= \sum_{k=1}^{\min(m_1, m_2)} {m_1 \choose k} {m_2 \choose k}
& =\hpg21{-m_1,-m_2}{1}{1}-1 \nonumber \\
&= {m_1+m_2 \choose m_1}-1. 
\end{align}
The binomial sum (from $k=0$) is evaluated as a special case of the Chu-Vandermonde formula \cite{Wiki}.
The bound $B(m,m)\in\Theta\big(4^m/\sqrt{m}\big)$ is well known. 
In \cite{vSt99}, a sharper upper bound in $\Theta\big((3\sqrt3/2)^m/\sqrt{m}\big)$ for the total number
of Nash equilibria is noted, and games of two players with $\Theta\big((\sqrt2+1)^m/\sqrt{m}\big)$
Nash equilibria are constructed. For $m=4$, the maximal number  \cite{McLP99} is 15.

Generally, our bound is
\begin{align} \label{eq:sms}
B(m_1,\ldots,m_S)=\sum_{k_1=1}^{m_1}\cdots \sum_{k_S=1}^{m_S} 
{m_1 \choose k_1}\cdots {m_S \choose k_S}E(k_1-1,\ldots,k_S-1).
\end{align}
First we prove an easier similar sum.
\begin{lemma} \label{th:sms}
$\displaystyle
\sum_{k_1=0}^{n_1}\cdots \sum_{k_S=0}^{n_S} 
{n_1 \choose k_1}\cdots {n_S \choose k_S}E(k_1,k_2,\ldots,k_S)=
\frac{(n_1+\ldots+n_S)!}{n_1!\cdots n_S!}.
$
\end{lemma}
\begin{proof}
We adopt the context of the card recreation described in the introduction.
The left-hand side counts block derangements in all subsets of $N=n_1+\ldots+n_S$ cards,
while the right-hand side counts all partitions of the $N$ cards to sets of
$n_1,\ldots,n_S$ cards. There is a bijection between the block derangements and the partitions,
where each partition (of the whole set of $N$ cards) is considered as a block derangement
(of a subset) after ignoring the cards that are dealt back to the same player.
We count $E(0,\ldots,0)=1$.
\end{proof}
We characterize the bound $B(m_1,\ldots,m_S)$ as a sum of multinomial coefficients.
That leads to the generating function for these numbers.
\begin{theorem} \label{th:sms2}
$\displaystyle
B(m_1,\ldots,m_S)=
\sum_{\ell_1=0}^{m_1-1}\cdots \sum_{\ell_S=0}^{m_S-1} 
\frac{(\ell_1+\ldots+\ell_S)!}{\ell_1!\cdots \ell_S!}.
$
\end{theorem}
\begin{proof}
We set $m_j=n_j+1$ and shift the summation indices in (\ref{eq:sms}):
\begin{align} \label{eq:sms2}
B(m_1,\ldots,m_S)=\sum_{k_1=0}^{n_1}\cdots \sum_{k_S=0}^{n_S} 
{n_1+1 \choose k_1+1}\cdots {n_S+1 \choose k_S+1}E(k_1,\ldots,k_S).
\end{align} 
Iterated use of  
${n+1 \choose k+1}={n \choose k}+{n \choose k+1}$
gives
\begin{align*}
{n_j+1 \choose k_j+1} &=
{n_j \choose k_j}+{n_j-1 \choose k_j}+\ldots+{k_j \choose k_j}+{k_j-1 \choose k_j}+\ldots 
\\ &=\sum_{\ell_j=0}^{n_j}  {\ell_j \choose k_j}.
\end{align*}
The binomial coefficients with $\ell_j<k_j$ are zero.  
After expanding the binomial coefficients in (\ref{eq:sms2}),
we sum up by the $k_j$'s first using Lemma \ref{th:sms}.
\end{proof}
\begin{corollary} The bound $B(m_1,\ldots,m_S)$ equals 
the coefficient to $x_1^{m_1}\cdots x_S^{m_S}$ 
in the multivariate Taylor expansion of
\begin{equation} \label{eq:genfb}
F_B(x_1,\ldots,x_S)=\frac{x_1\cdots x_S}{(1-x_1)\cdots(1-x_S)(1-x_1-\ldots-x_S)}
\end{equation}
at $(x_1,\ldots,x_S)=(0,\ldots,0)$.
\end{corollary}
\begin{proof}
The Taylor coefficients of $1/(1-\sigma_1)$ 
are the multinomial coefficients.
The factors $x_j/(1-x_j)$ represent their summation in Theorem \ref{th:sms2}.
\end{proof}

\begin{remark} \rm
Lemma \ref{th:sms} and Theorem \ref{th:sms2} can be proved using Legendre polynomials.
After substituting (\ref{eq:linc1})--(\ref{eq:linc2}) into Lemma \ref{th:sms} or (\ref{eq:sms2}) we recognize
\[
\int_0^\infty P_{n_1}(z)P_{n_2}(z)\cdots P_{n_S}(z)\exp(-z)dz,
\]
where, respectively,
\begin{align*}
P_{n}(z) = \sum_{k=0}^n (-1)^k {n\choose k} L_k(z) = \frac{z^n}{n!}  \ 
\mbox{  or } 
P_{n}(z) =\sum_{k=0}^n (-1)^k {n+1\choose k+1} L_k(z) =\sum_{j=0}^n\frac{z^j}{j!}
\end{align*}
by straightforward hypergeometric summation. 
Recall that \mbox{$\int_0^{\infty} \! z^N \! \exp(-z)dz=N!$.}
\end{remark}

\begin{remark} \rm
The generating function $F_{\rm N}$ in (\ref{eq:fnash}) is related to $F_B$ as follows:
\begin{equation*}
F_B={\cal B}^{-1} \left( \exp(-\sigma_1) \, {\cal B} F_{\rm N} \right).
\end{equation*}
Here $\cal B$ is the multivariate version of the Borel transform \cite{Wiki} that sends 
\mbox{$x_j^k\mapsto x^k_j/k!$}. 
Similarly, $1/(1-\sigma_1)={\cal B}^{-1} \left( \exp(-\sigma_1) \, {\cal B} F_{\rm M} \right)$
for the function $F_{\rm M}$ in (\ref{eq:fcard}).
\end{remark}

\begin{remark} \rm
The bound (\ref{eq:sms}) on the number of Nash equilibria 
is not sharp, just like (\ref{eq:max2}) in the $S=2$ case. In particular, $B(2,2,2)=16$ but a sharper bound is 9.
Instead of counting no more than 8 totally pure equilibria, 2 totally mixed and 6 other equilibria ($8+2+6=16$),
we count $4+2+3=9$, 
respectively.   The reason is that if someone plays a pure strategy (i.e., $k_j=1$), 
that strategy must be generically a unique best response.  
McLennan and McKelvey informed that they had randomly generated games 
(of 3 players with 2 options each) with indeed 9 Nash equilibria in total.
Generally, we can decrease the terms in (\ref{eq:sms}) with some $k_j=1$
by replacing one binomial coefficient $m_j\choose1$ with 1.
That would not affect the leading asymptotic term in Theorem \ref{th:basympt} below.
In the $S=2$ case, it is an important open problem whether the asymptotic
upper bound $\Theta\big((3\sqrt3/2)^m/\sqrt{m}\big)$ in \cite{vSt99} is sharp.
\end{remark} 

The $B$-numbers are combinatorially interesting, nevertheless. 
The integer sequence 
\[ 
B(m,m)={2m\choose m}-1
\]
appears in OEIS \cite{Sloan} as A030662.
The sequences $B(m,m,m)$, $B(m,m,m,m)$ are A144660, A144661, respectively.
Here we derive a few additional results, particularly on $B(a,b,c)$.
\begin{lemma}
\begin{itemize}
\item For positive integers $m_1,\ldots,m_S$,
\begin{equation} \label{eq:brec}
B(m_1,\ldots,m_S)= B(m_1-1)+\ldots+B(m_S-1)+1.
\end{equation}
The non-shifted $B$-arguments $m_j$ are skipped on the right-hand side,
as in Theorem $\ref{th:rec5}.$
\item We have
\begin{equation}  \label{eq:mcrec}
\sum_{k_1=0}^{m_1} \cdots \sum_{k_S=0}^{m_S} \big( 1-\#\{j:k_j<m_j\} \big)
\frac{(k_1+\ldots+k_S)!}{k_1!\cdots k_S!} =1. 
\end{equation}
Here only the summation term with all $k_j=m_j$ is positive.
\end{itemize}
\end{lemma}
\begin{proof}
If we multiply the generating function (\ref{eq:genfb}) by $1-x_1-\ldots-x_S$,
we get a multivariate Taylor series with all non-zero coefficients 
equal to 1. The multiplication of series relates the $B$-numbers as in the first formula. 
The second formula is obtained by increasing all $m_j$ by 1 in (\ref{eq:brec}),
and counting appearances of each multinomial coefficient.
\end{proof}

\begin{lemma}
For integers $a\ge 0$, $b\ge 0$, $c>0$,
\begin{align}    \label{eq:brec1}
B(a,b,c+1) - B(a,b,c-1) = & \, \frac{a\,b\,(a+b+2c)\,(a+b+c-1)!}{(a+c)\,(b+c)\,a!\,b!\,c!},\\
 \hspace{-20pt}  \label{eq:brec2}
 B(a+1,b+1,c-1)+B(a,b,c)= & \, \frac{(a+b+c)!}{(a+b+1)\,a!\,b!\,(c-1)!} -1, \\
  \label{eq:brec3} B(a,a,a)+\frac{1+(-1)^a}2= & \,
\sum_{k=0}^{a-1}(-1)^{a-k-1} \, \frac{7a+2}{2a+1}\,\frac{(3a)!}{(a!)^3}.
\end{align}
\end{lemma}
\begin{proof}  By Gosper's summation \cite{Wiki} 
\[
S(j,c):=\sum_{i=0}^{a-1} \frac{(i+j+c)!}{i!j!c!} = \frac{a\,(a+j+c)!}{(j+c+1)\,a!\,j!\;c!}.
\]
The term $S(j,c)$ is not Gosper-summable with respect to $j$. 
But \mbox{$S(j,c)+S(j,c-1)$} is Gosper-summable: $S(j,c)+S(j,c-1)=T(j+1)-T(j)$ with
\[
T(j)=\frac{a\,j\,(a+j+2c)\,(a+j+c-1)!}{(a+c)\,(j+c)\,a!\,j!\,c!}.
\]
Telescoping summation gives the first formula. 
Formula (\ref{eq:mcrec}) with $(m_1,m_2,m_3)=(a,b,c-1)$ gives
\[
1+B(a+1,b+1,c-1)+B(a,b,c)= \sum_{i=0}^{c-1} \frac{(a+b+i)!}{a!b!i!}.
\]
The summation is actually the same as $S(j,c)$, giving the second formula.
Combining (\ref{eq:brec2})  with a shifted version of (\ref{eq:brec1}) gives 
\begin{align*} \label{eq:brec4}
B(a+1,b+1,c+1)+ &\, B(a,b,c)= \\
& \frac{(a+b+c+1)^2(a+b+c+2)+abc}{(a+b+1)(a+c+1)(b+c+1)}
\,\frac{(a+b+c)!}{a!b!c!} -1.
\end{align*}
This specializes to the recurrence
\begin{align}
B(a+1,a+1,a+1)+B(a,a,a)= \frac{7a+2}{2a+1}\,\frac{(3a)!}{(a!)^3} -1.
\end{align}
The initial condition $B(0,0,0)=1$ leads to the last formula. 
\end{proof}
Recurrence (\ref{eq:brec4}) gives the ``diagonal"  generating function
\begin{equation} \label{eq:brec}
\sum_{m=0}^{\infty} B(m,m,m)\,x^m=
\frac1{1+x}\sum_{k=0}^{\infty} \frac{7k+2}{2k+1}\,\frac{(3k)!}{(k!)^3} \,x^k
-\frac1{1-x^2}.
\end{equation}

\section{Asymptotics}
\label{sec:asympt}

Asymptotics of coefficients of multivariate functions can be computed using the machinery
developed in \cite{PW02}, \cite{PW04}, \cite{RW08}. 
We compute the asymptotics for $E(n,\ldots,n)$, $E(a,b,c)$, $E(a,b,c,d)$ and $B(m_1,\ldots,m_S)$. 
\begin{theorem} \label{th:asnn}
$\displaystyle 
E(n,n,\ldots,n) \sim \frac{\sqrt{S} \,(S-1)^{Sn+S-1}}{\big( 2S(S-2)\pi n \big)^{\frac{S-1}2}}.
$
\end{theorem}
\begin{proof}
Let $H$ denote the denominator of $F_{\rm M}$ in (\ref{eq:fcard}).
In the context of \cite[\S 3]{PW02} or \cite[\S 3]{RW08}, the set of contributing points
is determined by $x_i\partial H/\partial x_i=x_j\partial H/\partial x_j$ for $i\neq j$, and $H=0$.
Since all coordinates $x_i$ must be positive, we have $x_i=x_j$. 
From (\ref{eq:mev}) we determine $x_i=1/(S-1)$. 
We have thus one smooth contributing point.
Let $\partial_j$ denote the differentiation $\partial/\partial x_j$ and 
eventual evaluation at the contributing point.
For the Hessian matrix, we evaluate simple binomial sums:
\[
\partial_i H 
= -\left(\frac{S}{S-1}\right)^{S-2},
\qquad 
\partial_i \partial_j H= -2\left(\frac{S}{S-1}\right)^{S-3}.
\]
In \cite[Proposition 3.4]{RW08}, the off-diagonal entries equal $a=1-2/S$.
The result follows.
\end{proof}
Using factorials, we can write
\begin{align}
E(n,n,\ldots,n) \sim  \left( \frac{S-1}{\sqrt{S(S-2)}} \right)^{\!S-1}
\! \left( 1-\frac1S \right)^{\!Sn} \,  \frac{(Sn)!}{(n!)^S}.
\end{align}
For the Franel numbers, we have $E(n,n,n)\sim 2^{3n+1}/\sqrt3\pi n$ 
as attributed to Keane in  \cite{Sloan}. The next result specializes to the same
asymptotics of Franel numbers. 
\begin{theorem}
As all three $a,b,c$ approach infinity in a proportional  manner,
\begin{align*}
E(a,b,c)\sim 
\frac{2^{a+b+c+1}}{\pi\sqrt{2ab+2ac+2bc-a^2-b^2-c^2}}
\frac{a!\,b!\,c!}{(a+b-c)! (a-b+c)! (b-a+c)!}.
\end{align*}
\end{theorem}
\begin{proof}
Considering $E(\alpha n,\beta n,\gamma n)$ as $n\to\infty$,
the only contributing point  $(x_1,x_2,x_3)$ for the 
positive direction $(\alpha,\beta,\gamma)$ is 
\[ \left( 
\frac{(\alpha+\beta-\gamma)(\alpha-\beta+\gamma)}{2\alpha(\beta-\alpha+\gamma)},  
\frac{(\alpha+\beta-\gamma)(\beta-\alpha+\gamma)}{2\beta(\alpha-\beta+\gamma)},  
\frac{(\alpha-\beta+\gamma)(\beta-\alpha+\gamma)}{2\gamma(\alpha+\beta-\gamma)}
\right).
\]
This gives $E(\alpha n,\beta n,\gamma n)\sim Cq^n/n$, where
\begin{align*}
q & = \frac{2^{\alpha+\beta+\gamma}\,\alpha^\alpha\,\beta^\beta\,\gamma^\gamma}
{(\alpha+\beta-\gamma)^{\alpha+\beta-\gamma}(\alpha-\beta+\gamma)^{\alpha-\beta+\gamma}
(\beta-\alpha+\gamma)^{\beta-\alpha+\gamma}}, \\
C & =\frac{2}{\pi} \sqrt{\frac{\alpha\,\beta\,\gamma}
{(\alpha+\beta-\gamma)(\alpha-\beta+\gamma)(\beta-\alpha+\gamma)
(2\alpha\beta+2\alpha\gamma+2\beta\gamma-\alpha^2-\beta^2-\gamma^2)}}.
\end{align*}
This is equivalent to the statement.
\end{proof}
If some number in $a,b,c$ remains bounded (but positive), the asymptotic estimate is wrong by a factor.
In particular, compare with $E(1,a,a)=2a$ and $E(1,a,a+1)=a+1$. 
The entity 
\[
\frac14\sqrt{2ab+2ac+2bc-a^2-b^2-c^2}
\]
is the area of the Euclidean triangle with the sides $\sqrt{a},\sqrt{b},\sqrt{c}$.

For $S=4$, the contributing points are determined by an algebraic equation 
of degree 4. To hide algebraic roots, the singular variety swept 
by the varying directions $(\alpha_1,\alpha_2,\alpha_3,\alpha_4)$ 
can be parametrized:
\begin{align} \label{eq:uvw}
\hspace{-5pt} (\alpha_1:\alpha_2:\alpha_3:\alpha_4)= & \nonumber \\
& \hspace{-60pt} \big( w(u+v-w-1) : (1-w)(u+v+w-2) : u(v-1) : (u-1)v \big), 
 \\ \label{eq:xyzt}
(x_1,x_2,x_3,x_4)= & \left( \frac{w}{u+v-w-1}, \frac{1-w}{u+v+w-2}, 
\frac{v-1}{u}, \frac{u-1}{v} \right).
\end{align}
We need real $u>1$, $v>1$, $w\in(0,1)$ for real positive $x_1,x_2,x_3,x_4$.
For those positive directions $(\alpha_1,\alpha_2,\alpha_3,\alpha_4)$
with exactly one suitable pre-image $(u,v,w)$ under (\ref{eq:uvw}),
there is just one contributing point (\ref{eq:xyzt}) and
\begin{align}
& E(\alpha_1n,\alpha_2n,\alpha_3n,\alpha_4n)\sim \frac{
\big(x_1^{\alpha_1}x_2^{\alpha_2}x_3^{\alpha_3}x_4^{\alpha_4}\big)^{\xi n}
}{4(u+v-1)\sqrt{Kx_1x_2x_3x_4}\;(\pi\xi n)^{-3/2}}
\end{align}
with $\xi=u(v-1)/\alpha_3$,
$K=  w(1-w)\big((u-v)^2\!+u+v-2\big)+(u-1)(v-1)(u+v-1).$
This applies to the symmetric direction $\alpha_1=\alpha_2=\alpha_3
=\alpha_4$ with $u=v=3/2$, $w=1/2$, consistent with Theorem \ref{th:asnn}. 

We finish off with asymptotics for the $B$-numbers of \S \ref{sec:allnash}.
\begin{theorem} \label{th:basympt}
With $m_1,\ldots,m_S$ all approaching infinity in a proportional manner,
\[
B(m_1,\ldots,m_S)\sim \frac{m_1\cdots m_S}{(M-m_1)\cdots(M-m_S)}
\frac{M!}{m_1!\cdots m_S!}, 
\]
Here $M=m_1+\ldots+m_S$.
\end{theorem}
\begin{proof}
We consider asymptotics of $B(\alpha_1n,\ldots,\alpha_Sn)$ with positive $\alpha_1,\ldots,\alpha_S$.
Let $A=\alpha_1+\ldots+\alpha_S$. The possible contributing points are:
\begin{itemize}
\item An isolated smooth point  $(x_1,\ldots,x_S)=(\alpha_1/A,\ldots,\alpha_S/A)$.
\item Intersection points of any two (or more) hyperplanes $x_j=1$. 
\end{itemize}
At the isolated point, all $\partial_jH=-(A-\alpha_1)\cdots(A-\alpha_S)/A^S$.
The Hessian matrix has simplified entries, 
as $c_d$ gets multiplied by zero inside the brackets
in \cite[Theorem 3.3]{RW08}:
\[
h_{ij}=\frac{x_ix_j\partial_iH\,\partial_jH}{x_S^2(\partial_SH)^2}
=\frac{\alpha_i\alpha_j}{\alpha_S^2}, \quad
h_{jj}=\frac{x_j\partial_jH}{x_S\partial_SH}+\left(\frac{x_j\partial_jH}{x_S\partial_SH}\right)^2
=\frac{\alpha_j(\alpha_j+\alpha_S)}{\alpha_S^2}.
\]
The Hessian determinant equals
\[
\frac{\alpha_1^2\alpha_2^2\cdots \alpha_{S-1}^2}{\alpha_S^{2S-2}} \, \det \widetilde{H}
=\frac{A\,\alpha_1\alpha_2\cdots \alpha_{S-1}}{\alpha_S^S}, 
\]
where $\widetilde{H}$ is the matrix in Lemma \ref{th:edet} with $x_j=\alpha_S/\alpha_j$. 
The isolated point gives the contribution
\begin{equation} \label{eq:domc} 
\sim \frac{\sqrt{A\,\alpha_1\,\cdots\,\alpha_S}}
{(2\pi n)^{\frac{S-1}2}(A-\alpha_1)\cdots(A-\alpha_S)}
\left( \frac{A^A}{\alpha_1^{\alpha_1}\,\cdots\,\alpha_S^{\alpha_S}} \right)^n.
\end{equation}
Intersections of $k$ hyperplanes contribute  \cite{PW04} 
only if the complementary $S-k$ direction components $\alpha_j$ are zero.
We consider only $k=S$, thus the multiple point $x_1=\ldots=x_S=1$.
It contributes sub-exponential asymptotics, thus (\ref{eq:domc}) dominates.
After rewriting  (\ref{eq:domc}) in terms of $m_j$'s and factorials, we get the result.
\end{proof}
\begin{corollary} $\displaystyle
B(m,m,\ldots,m) \sim \frac{S^{Sm+\frac12}}{\big( 2\pi m \big)^{\frac{S-1}2}\,(S-1)^S}.$
\end{corollary}

\vspace{12pt}

{\bf Acknolwledgement.} The author is grateful to Ira Gessel for informing about the full extent 
of relevant MacMahon's contribution, and the relation to Laguerre polynomials.

\end{document}